\documentclass[twoside,10pt]{amsart}
\usepackage{graphicx,amsmath,amsthm,amssymb,latexsym,amsfonts,color}
\usepackage[bookmarksnumbered, colorlinks, plainpages]{hyperref}
\usepackage{algorithm2e}
\usepackage{booktabs}
\usepackage{multirow}

\newtheorem{theorem}{Theorem}[section]

\newtheorem{remark}[theorem]{\bf{Remark}}

\usepackage{cleveref}
\begin{document}
\date{}

\title{Solving linear systems of the form $(A + \gamma UU^T)\, {\bf x} = {\bf b}$ \\ by preconditioned iterative methods}
\thanks{{\scriptsize
		$^\ast$Corresponding author}}
\maketitle
\begin{center}
	
	 \textbf{\textbf{Michele Benzi}$^{\ast, \S}$} \textbf{and \textbf{Chiara Faccio}$^{\dagger}$}  \\ [0.2cm]
	$^\S${\small \textit{Scuola Normale Superiore, Piazza dei Cavalieri, 7, 56126, Pisa, Italy}}\\
	\texttt{e-mail: michele.benzi@sns.it}\\[0.2cm]
	$^{\dagger}${\small \textit{Scuola Normale Superiore, Piazza dei Cavalieri, 7, 56126, Pisa, Italy}}\\
	\texttt{e-mail: chiara.faccio@sns.it}\\[0.2cm]
\end{center}

\begin{abstract} We consider the iterative solution of large linear systems of equations in which the coefficient matrix is the
sum of two terms, a sparse matrix $A$ and a possibly dense, rank deficient matrix of the form $\gamma UU^T$, where $\gamma > 0$
is a parameter which in some applications may be taken to be 1.  The matrix $A$ itself
can be singular, but we assume that the symmetric part of $A$ is positive semidefinite and that $A+\gamma UU^T$ is nonsingular. 
Linear systems of this form arise frequently in fields like optimization, fluid mechanics, computational statistics, 
and others.  We investigate preconditioning strategies based on an alternating splitting approach combined with the use 
of the Sherman-Morrison-Woodbury matrix identity.   The potential of the proposed approach is demonstrated by means of
numerical experiments on linear systems from different application areas.
	
	
\end{abstract}

\pagestyle{myheadings}
\markboth{\rightline {\scriptsize  Michele Benzi and Chiara Faccio}}
{\leftline{\scriptsize Preconditioning techniques for a class of linear systems
	}}
	
	\bigskip
	\bigskip

\section{Introduction}\label{sec1}

A problem that frequently arises in large-scale scientific computing is the solution of linear systems 
of the form 
\begin{equation}\label{eq1}
 (A + \gamma UU^T)\,  x =  b\,,
\end{equation}
where $A\in \mathbb{R}^{n\times n}$, $U\in \mathbb{R}^{n\times k}$, $\gamma > 0$ and ${\bf b}\in \mathbb{R}^n$ are given.
We make the following assumptions:

\begin{itemize}
\item The matrix $A$ is positive semidefinite, in the sense that the symmetric matrix $A+A^T$ is positive semidefinite
(sometimes, the term {\em semipositive real} is used).
\item The matrix $A + \gamma UU^T$ is nonsingular; that is, ${\rm Ker} (A)\cap {\rm Ker} (U^T) = \{ {\bf 0} \}$.
\item The number of columns $k$ of $U$ satisfies  $k< n$ (and often $k\ll n$).
\item Forming $A+\gamma UU^T$ explicitly would lead to loss of sparsity/structure and should be avoided.
\end{itemize}

Linear systems of the form (\ref{eq1}) with such features arise for instance in the solution of the augmented Lagrangian 
formulation of saddle point problems, in the solution of reduced KKT systems from interior point methods in constrained
optimization, and in the solution of sparse-dense least-squares problems. 
 
In principle, approaches based on the Sherman-Morrison-Woodbury (SMW) matrix identity (see \cite{GVL4}) could
be used to solve (\ref{eq1}), assuming that systems with coefficient matrix $A$ can be solved efficiently. In
this paper we focus mostly on situations where such an approach is not viable; for example, $A$ can be singular,
and/or the problem size is too large for linear systems with $A$ to be solved accurately. Nevertheless, the
SMW identity will play an important role in this paper, albeit not applied directly to (\ref{eq1}). 

Our focus is on the construction of preconditioners tailored to problem (\ref{eq1}), to be used in conjunction with
Krylov subspace methods. Since Krylov methods only need the coefficient matrix in the form of matrix-vector products,
 it is not necessary to explicitly add the two terms comprising the matrix $A + \gamma UU^T$, as long as matrix-vector
products involving the matrices $A$, $U$ and $U^T$ can be performed efficiently. Our main goal,  then, is to develop preconditioners
that can be set up without explicitly forming $A+\gamma UU^T$, but working only with $A$, $U$ and $U^T$. The preconditioners must be
inexpensive to construct and to apply, and effective at producing fast convergence of the preconditioned Krylov method.
Robustness with respect to $\gamma$ is also highly desirable.
%A further requirement is that for symmetric positive definite (SPD) systems the resulting preconditioner is also SPD,
%so that it can be used with the conjugate gradient method.

The remainder of the paper is organized as follows. In section \ref{sec2} we discuss motivating examples for the proposed solution
techniques,
which will be described in section \ref{sec3}. In section \ref{sec4} we briefly review some related work, while in section \ref{sec5}
we present some estimates on the eigenvalues of preconditioned matrices. Numerical experiments aimed at illustrating the performance of the
proposed solvers are presented in section \ref{sec6}; conclusions and suggestions for future work are given in section \ref{sec7}.

\section{Motivation}\label{sec2}
Large linear systems  of the form (\ref{eq1}) arise frequently in scientific computing.  Examples include:

%\vspace{0.2in}
\begin{itemize}
\item Augmented Lagrangian methods for PDE-related saddle point problems \cite{BB22,BO2006,BW13,F19,AL,Golub};
\item Solution of KKT systems in constrained optimization \cite{NW06};
\item Solution of sparse-dense least squares problems \cite{ST18,ST19,ST22};
%\item Certain types of integro-differential equations \cite{?,LinYang13};
\item Solution of PDEs modeling almost incompressible materials \cite[Ch.~8]{BBF};
\item Numerical solution of PDEs with nonlocal BC's \cite{Hu}.
\end{itemize}

Another situation where systems of the form (\ref{eq1}) may arise is when solving singular linear systems with a known kernel.
%, such as the Neumann problem 
%for for elliptic PDEs or systems involving the graph Laplacian.}

Next, we describe in some detail linear systems of the form (\ref{eq1}) from the first three of these applications.

\subsection{Linear systems from the augmented Lagrangian formulation}\label{2.1}

Consider the saddle point problem

\begin{displaymath}
{\mathcal A}\,  {\bf x} = \begin{bmatrix} A & B^T \\ B &  0 \end{bmatrix} \begin{bmatrix} u \\ p \end{bmatrix} = \begin{bmatrix}
f \\ g \end{bmatrix} = {\bf f}.
\end{displaymath}

\vspace{0.1in}
Such systems arise frequently from the finite element discretization of systems of PDEs, such as 
for example the Stokes equations, the Oseen problem (obtained from the steady Navier-Stokes equations
via Picard linearization), or from first-order system formulations of second-order elliptic PDEs; see, e.g., 
\cite{BBF,ESW14}.
A powerful approach to solve such systems is the one based on the augmented Lagrangian formulation \cite{AL}.
This method is  also widely used for solving constrained optimization problems \cite{NW06}.
The idea is to replace the original saddle point problem with an equivalent one of the form:

\begin{displaymath}
{\mathcal A}_\gamma \,  {\bf x} = \begin{bmatrix}
A + \gamma B^TW^{-1}B & B^T \\ B &  0
\end{bmatrix} \begin{bmatrix}
u \\ p \end{bmatrix} = \begin{bmatrix}
\hat  f \\ g \end{bmatrix} = \hat {\bf f}\,,
\end{displaymath}

\vspace{0.1in}
\noindent where $\gamma > 0$ and $\hat f := f + \gamma B^T W^{-1}g$. Here $W$ is usually  diagonal and positive definite. 
In the setting of finite element models of fluid flow, 
 $W$ is often the diagonal of the (pressure) mass matrix.
This new, augmented system is then solved by a Krylov subspace method with preconditioner

\begin{equation} \label{AL}
{\mathcal P}_\gamma = 
\begin{bmatrix}
A + \gamma B^TW^{-1}B & B^T \\
0 &  -\gamma^{-1} W\,.
\end{bmatrix}
\end{equation}

\vspace{0.1in}
The convergence of the preconditioned iteration is very fast  independent
of parameters like the mesh size and (for the Oseen problem) the viscosity, especially for large $ \gamma$ (see \cite{BO2006,F19});
to be practical, however, the preconditioner must be applied  inexactly. Evidently, the only difficulty in
applying the preconditioner is the solution of linear systems associated with the (1,1) block, i.e., 
a linear system with coefficient matrix $A + \gamma B^TW^{-1}B $ must be solved at each application
of the preconditioner. 
This linear system is of the form (\ref{eq1}) with $U = B^TW^{-1/2}$. 
%The matrix is nonsingular if and only if  ${\text Ker} (A) \cap {\text Ker} (B) = \{0\}$.
Here $A$ is sparse, often block diagonal, and positive definite (or $A+A^T$ is).  
 Forming
$A + \gamma B^TW^{-1}B$ explicitly leads to loss of structure, and depending on the discretization used 
the resulting matrix can be considerably less sparse than $A$. 
The condition number increases with  $\gamma$, and solving this system is the main
challenge associated with the augmented Lagrangian approach; hence, the need 
to develop efficient iterative methods for it. In
\cite{BO2006} and \cite{F19}, specialized geometric multigrid methods have been
developed for this task. While these methods have proven efficient, they suffer from
the limitations of geometric multigrid methods, primarily the fact that they are tied to very 
specific types of meshes and discretizations.  Here we consider algebraic approaches that can
be applied to very general situations.  We note that our preconditioned iteration being non-stationary,
it will require the (inexact) augmented Lagrangian preconditioner  ${\mathcal P}_\gamma $ to
be used as a (right) preconditioner for Flexible GMRES \cite{FGMRES}.
For symmetric problems, the Flexible Conjugate Gradient method may also be viable under certain
conditions \cite{Notay}. 

The solution of linear systems of the form (\ref{eq1}) is also required by the 
Relaxed Dimensional Factorization (RDF) preconditioner \cite{RDF,RDF16}, which has been developed in
particular for the Oseen problem. Here $A$ is the discretization of a (scalar) convection-diffusion operator
and $U^T$ represents the discretization of the partial derivative with respect to one of the space variables.
For 3D problems, three such linear systems must be solved at each application of the preconditioner.

\subsection{Schur complement systems arising from Interior Point methods}

The solution of (smooth) constrained minimization problems by Interior Point (IP) methods (see \cite{NW06}) leads
to sequences of linear systems of the form

\begin{displaymath}
{\mathcal A}\,  {\bf x} = \begin{bmatrix}
H & -C^T & 0 \\
C &  0 & -I_k \\
0 & Z & \Lambda
\end{bmatrix} \begin{bmatrix}
\delta x \\ \delta \lambda \\ \delta z
\end{bmatrix} = \begin{bmatrix}
-r_1 \\ - r_2 \\ - r_3
\end{bmatrix} =  {\bf f}.
\end{displaymath}

\vspace{0.1in}
Here $H=H^T$ is the $n\times n$ Hessian of the objective function at the current point $\bar x$, $C$ is the $k\times n$ Jacobian of the constraints at the same point,
and $Z$ and $\Lambda$ are diagonal, positive definite $k\times k$ matrices associated with the current values of the Lagrange multipliers $\bar \lambda$
and slack variables $\bar z$, respectively. The right-hand sides contains the nonlinear residuals.
The variable $\delta z$ can easily be obtained using the last equation:

\begin{displaymath}
\delta z = - \Lambda ^{-1} (r_3 + Z \delta \lambda)
\end{displaymath}
\vspace{0.1in}
and substituted into the second (block) equation. 
This yields the reduced system

\begin{displaymath}
\begin{bmatrix}
H & -C^T \\
C &  \Lambda^{-1} Z
\end{bmatrix} \begin{bmatrix}
 \delta x \\ \delta \lambda
\end{bmatrix} = \begin{bmatrix}
-r_1 \\ -r_2 - \Lambda^{-1} r_3
\end{bmatrix}.
\end{displaymath}

\vspace{0.1in}
Eliminating $\delta \lambda$ leads to the  fully reduced (Schur complement) system

\begin{equation}\label{KKT}
 (H + C^T Z^{-1}\Lambda C) \delta x = - r_1 - C^TZ^{-1} (r_3 + \Lambda r_2)=:b.
 \end{equation}

\vspace{0.1in}
After solving for $\delta x$, the other unknowns $\delta \lambda$ and $\delta z$ are readily obtained.
This system is of the form (\ref{eq1}) with 
$A = H$,  $U = C^T(Z^{-1}\Lambda)^{1/2}$ and $\gamma = 1$. The coefficient matrix is nonsingular if and only if
$\text {Ker} (H) \cap \text{Ker} (C) = \{0\}$. 
The Hessian is usually positive (semi)definite, sparse and possibly structured. 
The coefficient matrix is nonsingular if and only if
$\text {Ker} (H) \cap \text{Ker} (C) = \{0\}$. 
Especially for very large problems, forming
$H + C^T Z^{-1}\Lambda C$ explicitly is generally undesirable.  Instead, we propose to solve the 
fully reduced system with PCG or another Krylov method using a suitable (algebraic) preconditioner. 

\subsection{Sparse-dense least squares problems}

Consider a large linear least squares (LS) problem of the form
\begin{displaymath}
\| B x - c \|_2 = \min,
\end{displaymath} 
where $B \in \mathbb{R}^{m\times n}$,  and $c\in \mathbb{R}^m$.  We assume that $B$ has full column rank 
and that it has the following structure:
\begin{displaymath}
B = \begin{bmatrix}
B_1\\B_2
\end{bmatrix}\,, \quad B_1 \in \mathbb{R}^{(m-k)\times n}, \quad B_2\in \mathbb{R}^{k\times n},
\end{displaymath} 
where $B_1$ is  sparse and $B_2$ is  dense.  
%Then the normal equations can be written as the $n\times n$ system 
Then the LS problem is equivalent to the $n\times n$ system of normal equations:
\begin{equation} \label{normal}
B^T B x =  (B_1^T B_1 + B_2^TB_2) x = B^T c\,,
\end{equation}

\noindent which is of the form (1) with $ A = B_1^TB_1$, $U = B_2^T$, $\gamma = 1$ and $b = B^Tc$.
Once again, we would like to solve this system by  an iterative method, so the matrix $B^TB$ is never formed
explicitly. The main challenge is again
constructing an effective preconditioner. 

Recently, sparse-dense LS problems and various methods for their solution have been investigated in \cite{ST18,ST19,ST22}.

\section{The proposed method and its variants}\label{sec3}

In this section we first describe a stationary iterative method for solving (\ref{eq1}), then we develop
a more practical preconditioner based on this solver.  Although not strictly necessary, we assume that
the coefficient matrix $A_\gamma := A + \gamma UU^T$ is nonsingular for all $\gamma >0$. 
As we have seen in the previous section, in many applications 
the matrix $A$ (or $A+A^T$) is usually at least positive semidefinite,
and we will make this assumption.  Then the nonsingularity of $A_\gamma$ is equivalent to the condition
\begin{displaymath}
\textnormal{Ker}\, (A) \cap \textnormal{Ker}\, (U^T) = \{0\}\,,
\end{displaymath}
hence if $A$ is symmetric positive semidefinite and the above condition holds, then $A_\gamma$ is symmetric positive definite (SPD) for all 
$\gamma > 0$.

\vspace{0.1in}
When $A$ is nonsingular, one could use the Sherman-Morrison-Woodbury (SMW) formula to
solve (1), but this can be expensive for large problems. Recall that SMW states that
\begin{displaymath}
(A+ \gamma UU^T)^{-1} = A^{-1} - \gamma A^{-1}U(I_k + \gamma U^TA^{-1}U)^{-1}U^TA^{-1},
\end{displaymath} 
hence  $k+1$ linear systems of size $n\times n$  with coefficient matrix $A$ and an additional $k\times k$ system must be solved ``exactly". 
Another possibility would be to build
preconditioners based on the SMW formula, where the action of $A^{-1}$ is replaced by some inexpensive approximation,
but our attempts in this direction were unsuccessful. Also, $A$ is frequently singular.

\vspace{0.1in}
When $k$ is small (say, $k=10$ or less), then any good preconditioner for $A$ (or $A+\alpha I_n$, $\alpha > 0$, if
$A$ is singular) can be expected to give good results. In fact, using  a Krylov method preconditioned with $A^{-1}$ yields convergence in
at most $k+1$ steps, hence convergence should be fast if a good approximation of the action of $A^{-1}$ is available.  
However, if $k$ is in the hundreds (or larger), this approach is not appealing. 

\vspace{0.1in}
Hence, it is necessary to take into account both $A$ and $\gamma UU^T$ when building the preconditioner.
We do this by forming a suitable product preconditioner, as follows.

\vspace{0.1in}
Let $\alpha > 0$ be a parameter and consider the two splittings
\begin{displaymath}
A+ \gamma UU^T = (A + \alpha I_n) - (\alpha I_n - \gamma UU^T)
\end{displaymath} 
and
\begin{displaymath}
A +\gamma UU^T= (\alpha I_n + \gamma UU^T) - (\alpha I_n - A).
\end{displaymath} 

\vspace{0.1in}
Note that both $A+\alpha I_n$ and $\alpha I_n + \gamma UU^T$ are invertible under our assumptions. Let $x^{(0)} \in \mathbb{R}^n$ and consider the  alternating iteration

\begin{equation}\label{iteration}
\begin{cases}
\, (A + \alpha I_n) \, x^{(k+1/2)} = (\alpha I_n - \gamma UU^T) \, x^{(k)} + b\,, \\
\, (\alpha I_n + \gamma UU^T)\, x^{(k+1)} = (\alpha I_n - A) \, x^{(k+1/2)} + b\,,
\end{cases}
\end{equation}
with $k=0,1,\ldots$ This alternating scheme is analogous to that of other well-known
iterative methods like ADI \cite{ADI}, HSS \cite{HSS}, MHSS \cite{HSS}, RDF \cite{RDF}, etc.
Similar to these methods, we have the following convergence result.

\begin{theorem}\label{conv_thm}
If $A + A^T$ is positive definite, the sequence  $\{ x^{(k)} \}$ defined by (\ref{iteration}) converges,
as $k\to \infty $, to the unique solution of equation (1), for any choice of $x^{(0)}$ and for all
$\alpha > 0$.
\end{theorem}
\begin{proof}
First we observe that under our assumptions the linear system (\ref{eq1}) has a unique solution. 
Eliminating the intermediate vector $x^{(k+1/2)}$, we can rewrite (\ref{iteration}) as a one-step 
stationary iteration of the form
%\begin{equation}\label{iteration2}
\begin{displaymath}
x^{(k+1)} = T_\alpha x^{(k)} + d \,, \qquad k = 0, 1, \ldots ,
\end{displaymath}
where $d$ is a suitable vector and
the iteration matrix $T_\alpha$ is given by
\begin{displaymath}
T_\alpha = (\alpha I_n + \gamma UU^T)^{-1}(\alpha I_n - A)(\alpha I_n + A)^{-1}(\alpha I_n - \gamma UU^T).
\end{displaymath} 
This matrix is similar to 
\begin{displaymath}
 \hat T_\alpha = (\alpha I_n - A)(\alpha I _n + A)^{-1}(\alpha I_n - \gamma UU^T)(\alpha I_n + \gamma UU^T)^{-1}\,,
\end{displaymath}
hence the spectral radius of $T_\alpha$ satisfies
\begin{displaymath}
 \rho (T_\alpha) = \rho (\hat T_\alpha) \le \|\hat T_\alpha\|_2 \le
\|(\alpha I_n - A)(\alpha I_n + A)^{-1}\|_2 \|(\alpha I_n - \gamma UU^T)(\alpha I_n + \gamma UU^T)^{-1}\|_2\,.
\end{displaymath}
The first norm on the right-hand side is strictly less than 1 for $\alpha >0$ since the symmetric part of $A$ is positive definite
(this result is sometimes referred to as ``Kellogg's Lemma", see  \cite[page 13]{Marchuk}), while the second one is obviously
equal to 1 for all $\alpha > 0$, since $UU^T$ is symmetric positive semidefinite and singular. 
Therefore $\rho (T_\alpha) < 1$ and the iteration is convergent
for all $\alpha > 0$. 
\end{proof}

\begin{remark} \label{Rm1}
If $A+A^T$ is only positive semidefinite, then all we can say is that $\rho (T_\alpha) \le 1$ and that 1 is 
not an eigenvalue of $T_\alpha$, for all $\alpha >0$.  
In this case we can still use the iterates $x^{(k)}$ to construct a convergent sequence that approximates the unique
solution of (\ref{eq1}); indeed, it is enough to replace $T_\alpha$ with $(1-\beta)\, I_n + \beta \, T_\alpha$, with $\beta \in (0,1)$,
to obtain a convergent sequence; see \cite[p.~27]{BG2004}. 
\end{remark}

\vspace{0.1in}
In practice, the stationary iteration (\ref{iteration}) may not be very efficient. It requires exactly solving two linear systems
with matrices $A+\alpha I_n$ and $\alpha I_n + \gamma UU^T$ at each step; even if these two systems are generally easier to solve than 
the original system (\ref{eq1}), convergence can be slow and the overall method expensive. 
To turn this into a practical method, we will use it as a preconditioner for a Krylov-type method
rather than as a stationary iterative scheme.  This will also allow inexact solves.

\vspace{0.1in}
To derive the preconditioner we eliminate $x^{(k+1/2)}$ from (\ref{iteration}) and write the iterative scheme as the fixed-point iteration
\begin{displaymath}
x^{(k+1)} = T_\alpha x^{(k)} + c =  (I_n - P_\alpha ^{-1}A_\gamma ) x^{(k)} + P_\alpha ^{-1}b\,.
\end{displaymath} 
 An easy calculation (see also \cite{BS97}) reveals that the preconditioner $P_\alpha $ is 
given, in factored form, by

\begin{equation}\label{prec}
 P_\alpha  = \frac{1}{2\alpha} (A + \alpha I_n)(\alpha I_n + \gamma UU^T).
 \end{equation}

\vspace{0.1in}
The scalar factor $\frac{1}{2\alpha}$ in (\ref{prec}) is immaterial for preconditioning, and can be ignored in practice.
Applying this preconditioner within a Krylov method requires, at each step, the solution of two linear systems
with coefficient matrices  $A+\alpha I_n$ and $\alpha I_n + \gamma UU^T$.
Consider first solves involving $A + \alpha I_n$.  If $A$ is sparse and/or structured (e.g., block diagonal, banded, Toeplitz, etc.),
then so is $A+\alpha I_n$.  If expensive,
exact solves with $A+ \alpha I_n$ can be replaced, if necessary, with inexact solves
using either a good preconditioner for $A+\alpha I_n$, such as an incomplete factorization, or a fixed number (one may be enough)
of cycles of some multigrid method if $A$ originates from a second-order elliptic PDE.
We note here a typical trade-off: larger values of $\alpha$ make solves with $A+\alpha I_n$ easier (since the matrix becomes better conditioned
and more diagonally dominant), but may degrade the performance of the preconditioner $P_\alpha$.
In practice, we found that high accuracy is not required in the solution of linear systems associated with $A+\alpha I_n$.

\vspace{0.1in}
On the other hand, numerical experiments suggest that the solution of linear systems involving $\alpha I_n + \gamma UU^T$ is more critical.
Note that this matrix is SPD for all $\alpha > 0$, but ill-conditioned for small $\alpha$ (or very large $\gamma$).
The Sherman-Morrison-Woodbury formula yields
\begin{equation}\label{SMW}
(\alpha I_n + \gamma UU^T)^{-1} =
\alpha^{-1} I_n - \alpha^{-1} \gamma U(\alpha I_k + \gamma U^TU)^{-1}U^T.
\end{equation}

The main cost is the solution at each step of a $k\times k$ linear system with matrix $\alpha I_k + \gamma U^TU$, 
which can be performed by Cholesky factorization (computed once and for all at the outset) or possibly
by a suitable inner PCG iteration or maybe an (algebraic) MG method.
Formula (\ref{SMW}) shows why this $k\times k$ linear system must be solved accurately: any error affecting 
$(\alpha I_k + \gamma U^TU)^{-1}$, and therefore $U(\alpha I_k + \gamma U^TU)^{-1}U^T$, will be be amplified
by the factor $\gamma/\alpha$, which will be quite large for small $\alpha$ and large or even moderate values of $\gamma$.
Note that for linear systems arising from the augmented Lagrangian method applied to incompressible flow problems, 
the matrix $\alpha I_k + \gamma U^TU$ is  a (shifted and scaled)
discrete pressure Laplacian. Also, this matrix remains constant in the course of 
the numerical solution of the Navier--Stokes equations using  Picard or Newton
iteration, whereas the matrix $A$ changes. 
 Hence, the cost of a  Cholesky
factorization of $\alpha I_k + \gamma U^TU$ can be amortized over many nonlinear (or time) steps.   
Similar observations apply if one uses an algebraic multigrid (AMG) solver instead of a direct factorization,
 in the sense that the preconditioner set-up needs to be done only once. 

\subsection{Some variants}\label{sec_sub}
Building on the main idea, different variants of the preconditioner can be envisioned. 
If $A$ happens to be nonsingular and linear systems with $A$ are not too difficult to solve (inexactly or perhaps even exactly),
 then it may not be necessary to shift $A$, leading to a preconditioner
of the form
\begin{displaymath}
\hat P_\alpha = A(\alpha I_n + \gamma UU^T).
\end{displaymath} 
Note that Theorem \ref{conv_thm}, however, is no longer applicable in general. 

\vspace{0.1in}
When $A$ is symmetric positive semidefinite and the usual assumption
$ \textnormal{Ker}\, (A) \cap \textnormal{Ker}\, (U^T) = \{0\}$ holds (so that $A_\gamma$ is SPD for $\gamma >0$), one would like to solve system (\ref{eq1})
using the preconditioned conjugate gradient (PCG) method. Unless $A$ and $UU^T$ commute, however, the preconditioner (\ref{prec}) 
is nonsymmetric, and the preconditioned matrix $P_\alpha^{-1} A_\gamma$ is generally not symmetrizable. 
 In this case we can consider a symmetrized version of the preconditioner,
for example
\begin{equation}\label{symm}
 P_{\alpha}^S = \frac{1}{2\alpha}  L\, (\alpha I_n + \gamma UU^T)\, L^T\,,
 \end{equation}
where $L$ is the Cholesky (or incomplete Cholesky) factor of $A+ \alpha I_n$ (or of $A$ itself  if $A$  is SPD and not very
ill-conditioned).  Again, Theorem \ref{conv_thm} no longer holds, in general.

\vspace{0.1in}
In some cases (but not always) the performance of the method improves if $A_\gamma$ is diagonally
scaled so that it has unit diagonal prior to forming the preconditioner. 
 Note that the matrix
\begin{displaymath}
D_\gamma : = \text{diag}\, (A + \gamma UU^T)
\end{displaymath} 
can be easily computed:  
\begin{displaymath}
(D_\gamma)_{ii} = a_{ii} + \gamma \|u_i^T\|_2 ^2\,,
\end{displaymath}
where $u_i^T$ is the $i$th row of $U$. 
It is easy to see that applying the preconditioner to the diagonally scaled matrix
$D_\gamma ^{-1/2} A_\gamma D_\gamma ^{-1/2}$ is mathematically equivalent to using the modified preconditioner
\begin{displaymath}
(A +\alpha D_\gamma)D_\gamma ^{-1} (\alpha D_\gamma + \gamma UU^T)
\end{displaymath} 
on the original matrix. We emphasize that whether this diagonal scaling is beneficial or not
appears to be strongly problem-dependent. Numerical experiments indicate that such scaling can lead to 
a degradation of performance in some cases. 
Clearly, different SPD matrices (other than the diagonal of $A_\gamma$) could be used for $D_\gamma$. 

One can also conceive two-parameter variants, $P_{\alpha, \beta}  = (A + \alpha I_n) (\beta I_n + \gamma UU^T)$, but we shall
not pursue such generalizations here. 

Finally, we observe that the extension to the complex case (under the obvious assumptions)  is straightforward.

\section{Related work}\label{sec4}
There seems to have been relatively little work on the development of specific solvers for linear systems of the form (\ref{eq1}).
The few papers we are aware of either advocate for the use of the SMW formula directly applied to (\ref{eq1}), or treat
specialized methods for very specific situations. 

In the recent papers \cite{Lu1,Lu2}, the author addresses the solution of linear systems closely related to (\ref{eq1}) 
by means of an auxiliary space preconditioning approach. This approach is specific to finite element discretizations
of PDE problems involving the De Rham complex and the Hodge Laplacian, such as those arising from the solution
of the curl-curl formulation of Maxwell's equations. 

Potentially relevant to our approach is the work on robust multigrid preconditioners for  finite element discretizations of the operator
 $\mathcal{ L }=  I - \gamma\,  \text {grad}\,  \text {div}$,
defined on the space $H(\text{div})$. Indeed, for the type of incompressible flow problems described in
section \ref{2.1} the matrix  $\alpha I_n + \gamma UU^T$ may be regarded as a discretized version of this operator, which
plays an important role in several applications; see, e.g.,  \cite{Hdiv97,Hdiv00,MW2011}.  In cases where direct use of the SMW 
 formula for solving
linear systems with matrix $\alpha I_n + \gamma UU^T$ is not viable, for example in very large 3D situations where Cholesky
factorization of a $k\times k$ matrix may be too expensive, such multigrid methods could be an attractive alternative in view of their
robustness and fast convergence. 

Finally, we mention the work in \cite{Marin}, although it concerns a somewhat different type of problem, namely, linear systems with
coefficient matrix of the form $A + UCU^T$ with $A=A^T$ and $C = - C^T$. Here $U \in \mathbb{R}^ {n\times k}$ and
$C\in \mathbb{R}^{k\times k}$.  We point out that 
for this type of problem, our approach reduces to the well-known HSS  method (or preconditioner), see \cite{HSS}.

\section{Eigenvalue bounds}\label{sec5}
%\section{Sme remarks on the choice of $\alpha$.} \label{sec5}

As is also the case for other solvers and preconditioners like ADI or HSS, the choice of $\alpha$ is important for the
success of the method. It is not easy to determine an ``optimal" or even good value of $\alpha$ a priori. Usually it is necessary
to resort to heuristics, one of which will be discussed in the next section.  Here we attempt to shed some light on the effect
of $\alpha$ on the spectrum of the preconditioned matrix $P_\alpha^{-1} A_\gamma$.  
 Note that by virtue of Theorem \ref{conv_thm}
and Remark \ref{Rm1}, we know that the spectrum of $\sigma(P_\alpha ^{-1}A_\gamma)$
lies in the disk of center $(1,0)$ and radius 1 in the complex plane, for all $\alpha > 0$. In particular, all the eigenvalues have imaginary part 
bounded by 1 in magnitude.

First we consider the case where $A$ is SPD. Let $0<\lambda_1 \le \cdots \le  \lambda_n$ be the eigenvalues of $A$.
As shown in the proof of Theorem \ref{conv_thm}, the spectral radius of the iteration matrix 
$T_\alpha = I_n - P_\alpha^{-1} A_\gamma$ of the
stationary iteration (\ref{iteration}) satisfies 
\begin{displaymath}
\rho (T_\alpha) \le \| (\alpha I_n - A) (\alpha I_n + A)^{-1}\|_2 = \max_{1\le i \le n} \frac{|\alpha - \lambda_i|}{|\alpha + \lambda_i|}\,,
\end{displaymath}
for all $\alpha > 0$.  The upper bound on $\rho (T_\alpha)$ is minimized, as is well known, taking $\alpha = \sqrt{\lambda_1 \lambda_n}$.  
A similar observation, incidentally, has been made  for the HSS method in \cite{HSS}, with the eigenvalues of $H= \frac{1}{2} (A + A^T)$
playing the role of the $\lambda_i$'s. This choice of $\alpha$ is completely independent of $\gamma$ and $U$, and thus it is not likely
to be always a good choice, especially when the method is used as a preconditioner rather than as a stationary solver.

In order to state the next result, we note that there is no loss of generality if we assume that
$\|A\|_2 = 1$ and $\| U\|_2 = 1$, since we can always divide both sides of (\ref{eq1}) by $\|A\|_2$,
replace $\gamma$ with $\tilde \gamma := \gamma \|U\|_2^2/\|A\|_2$ and $U$ with $\tilde U := U/\|U\|_2$. 

\begin{theorem}\label{thm2}
Let $A$ be such that $A+A^T$ is positive definite. Assume that $\|A\|_2 = \|U\|_2 =1$. Let $P_\alpha$ be given by (\ref{prec}). If $(\lambda, x)$ is a real eigenpair
of the preconditioned matrix $P_\alpha ^{-1} A_\gamma$, with $\|x\|_2 = 1$, then $\lambda \in [\mu, 2)$
 where 
 \begin{equation}\label{mu}
 \mu = \frac {\alpha\, \lambda_{\min} (A+A^T)}  {(1+\alpha)(\alpha+\gamma)}.
 \end{equation}\\
 If $(\eta, x)$ is an eigenpair of $A$ with $x\in \textnormal{Ker}\, (U^T)$, then $x$ is eigenvector of $P_\alpha ^{-1} A_\gamma$
 associated to the eigenvalue
 \begin{equation}\label{last}
 \lambda =   \frac{2\, \eta}{\eta + \alpha}
 \end{equation}
 (independent of $\gamma$). %Such eigenvalues are real if $A$ is symmetric (or if $\mu is real).
 %Finally, if $(\lambda, x)$ is an eigenpair with $x\in \textnormal{Ker} (A^T)$, then $U^Tx\ne 0$ and
% \begin{equation}\label{lastest}
 %\lambda = \frac{2}{1 + \frac{\alpha}{\gamma \|U^Tx\|_2^2}}
% \end{equation}
% (independent of $A$). Note that in this case $\lambda$ is real.
 \end{theorem}
 \begin{proof}
 The eigenpairs $(\lambda,x)$ of $P_\alpha ^{-1} A_\gamma$ (or, equivalently, of $A_\gamma P_\alpha^{-1}$) satisfy the
 generalized eigenvalue problem
 \begin{equation}\label{gep}
 (A + \gamma UU^T)\,  x = \frac{\lambda}{2\alpha} \,  (\alpha A + \gamma AUU^T + \alpha^2 I_n + \alpha\, \gamma\, UU^T)\,  x\,.
 \end{equation}
 Premultiplying by $x^*$ and using $x^*x=1$, we obtain
 \begin{equation}\label{lambda}
 \lambda = \frac{2\alpha\, (x^*Ax + \gamma \|U^Tx\|_2^2)}{\alpha \, x^*Ax +\gamma\,  x^*AUU^Tx + \alpha^2 + \alpha\, \gamma\, \|U^Tx\|_2^2} \,.
 \end{equation}
 \\
  If $\lambda$ is real then $x$ can be taken real and $x^*$ becomes $x^T$.  Clearly $0 < \lambda < 2$ as 
 an immediate consequence of Theorem \ref{conv_thm},
 which states that $|\lambda - 1| < 1$. To prove the lower bound on $\lambda$, note that a lower bound
 on the numerator in (\ref{lambda})  is given by $\alpha \,  \lambda_{\min} (A+ A^T)$,
 while an upper bound for the denominator is given by $\alpha + \gamma + \alpha^2 + \alpha \gamma = (1+\alpha)(\alpha + \gamma)$,
 yielding the value (\ref{mu}) for the lower bound on $\lambda$. 
 
If $U^Tx=0$ we immediately obtain (\ref{last}) from (\ref{lambda}). Note that all  such $\lambda$'s (if there are any) are necessarily real if $A=A^T$.
% Finally, assume $A^T x = 0$, then $x^*Ax = 0$ and $x^*AUU^Tx = 0$.  Thus, (\ref{lambda}) reduces to
% $$ \lambda = \frac {2\, \alpha\,  \gamma \|U^Tx\|_2^2}{\alpha^2 + \alpha\, \gamma\, \|U^Tx\|_2^2}\,$$
 %This shows that $U^Tx \ne 0$ since $\lambda\ne 0$.  Dividing the numerator and denominator  by $\alpha\, \gamma\, \|U^Tx\|_2^2$ we obtain (\ref{lastest}).
 \end{proof}
 
 \vspace{0.1in}
 \begin{remark}
 If we assume $A$ to be symmetric positive definite, then using (\ref{lambda}) one can easily establish the following  lower bound on the real part
 of the eigenvalues of $P_\alpha ^{-1} A_\gamma$, whether real or not:
 \begin{equation}\label{lower_bound}
 Re(\lambda) \ge \frac {2\, \alpha (\alpha + 1) (\alpha + \gamma) \,  \lambda_{\min}  (A) } {(\alpha +1)^2 (\alpha + \gamma)^2 + \gamma^2}\,.
\end{equation}
We computed the eigenvalues in several cases with an SPD matrix $A$ (see next section) and we 
found that usually the lower bound given by (\ref{mu}) yields a much better estimate of the smallest
real part of $\lambda \in \sigma(P_\alpha ^{-1}A_\gamma)$ than (\ref{lower_bound}), suggesting that the eigenvalue of smallest real part is actually real
in many cases.  
Our experiments confirm that when $A$ is SPD, the smallest eigenvalue of $P_\alpha ^{-1} A_\gamma$ is often real,
but this is not true in general.  
Note that all the bounds still hold if $A$ is singular, but give no useful information in this case. 
 \end{remark}
 
 \begin{remark}
 If $A$ is singular and if $(\lambda, x)$ is an eigenpair of $P_\alpha ^{-1} A_\gamma$ with $x\in \textnormal{Ker} (A^T)$, then $U^Tx\ne 0$ and
 %\begin{equation}\label{lastest}
 $$\lambda = \frac{2}{1 + \frac{\alpha}{\gamma \|U^Tx\|_2^2}}$$
% \end{equation}
(independent of $A$).  Indeed, we have $x^*Ax = 0$ and $x^*AUU^Tx = 0$.  Thus, (\ref{lambda}) reduces to
$$ \lambda = \frac {2\, \alpha\,  \gamma \|U^Tx\|_2^2}{\alpha^2 + \alpha\, \gamma\, \|U^Tx\|_2^2}\,$$
Clearly $U^Tx \ne 0$ since we are assuming that $A_\gamma$ is nonsingular.  Dividing the numerator and denominator  by $\alpha\, \gamma\, \|U^Tx\|_2^2$ we 
obtain the result. Note that such a $\lambda$, if it exists, is real.
% Finally, assume $A^T x = 0$, then $x^*Ax = 0$ and $x^*AUU^Tx = 0$.  Thus, (\ref{lambda}) reduces to
% $$ \lambda = \frac {2\, \alpha\,  \gamma \|U^Tx\|_2^2}{\alpha^2 + \alpha\, \gamma\, \|U^Tx\|_2^2}\,$$
 %This shows that $U^Tx \ne 0$ since $\lambda\ne 0$.  Dividing the numerator and denominator  by $\alpha\, \gamma\, \|U^Tx\|_2^2$ we obtain (\ref{lastest}).
\end{remark}

\begin{remark}
Theorem \ref{thm2} is of limited use for guiding in the choice of $\alpha$. It is easy to see that the lower bound (\ref{mu}) is maximized 
(when $A$ is nonsingular) by taking $\alpha =\sqrt{\gamma}$.  While such a choice of $\alpha$ may prevent the smallest
eigenvalue of the preconditioned matrix from getting too close to 0, in most cases such value of $\alpha$ is suboptimal. We also
note that the lower bound approaches zero as $\alpha \to 0$ and $\gamma \to \infty$, yet small values of $\alpha$ often
yield faster convergence, even for large values of $\gamma$, suggesting that a better clustering of the preconditioned spectrum is
achieved for smaller values of $\alpha$. 
One should also keep in mind that the result assumes that the preconditioner is applied  exactly, which is often not the case in practice,
and that eigenvalues alone may not be descriptive of the convergence of Krylov subspace methods like GMRES.
Nevertheless, setting $\alpha = \sqrt{\gamma}$ could be a reasonable choice in the absence of other information, provided
of course that the problem is scaled so that $\|A\|_2 = \|U\|_2 = 1$.
 \end{remark}
 
\vspace{0.1in}
We conclude this section with a result concerning the symmetrized preconditioner (\ref{symm}).

\begin{theorem} \label{thm3}
Let $A$ be SPD, $A_\gamma = A + \gamma UU^T$, with $\|A\|_2 = 1$ and $\|U\|_2 = 1$, and let $P_\alpha^S = 
\frac{1}{2\alpha} L\, (\alpha I_n + \gamma UU^T)\,L^T$
where $L$ is the Cholesky factor of $A+\alpha I_n$.  Then the eigenvalues $\lambda$ of the preconditioned matrix $(P_\alpha^S)^{-1} A_\gamma$ are all real
and lie in the interval
\begin{equation}\label{eig_symm}
\frac { 2\, \alpha \, \lambda_{\min} (A) } { (1 + \alpha) (\alpha + \gamma)} < \lambda < \frac { 2 + 2\, \gamma } {\lambda_{\min} (A) + \alpha} \,.
\end{equation}
\end{theorem}
\begin{proof}
That the eigenvalues of the preconditioned matrix are real (and positive) is an immediate consequence of the fact that both
the preconditioner $P_\alpha^S$ in (\ref{symm}) and the coefficient matrix $A_\gamma = A + \gamma UU^T$ are SPD. If $(\lambda, x)$
is an eigenpair of $(P_\alpha^S)^{-1}A_\gamma$ then 
\begin{displaymath}
\lambda = \frac{2\, \alpha \, \,x^T (A + \gamma UU^T) \, x}{x^T L (\alpha I_n + \gamma UU^T)L^T x} = 
\frac{2\, \alpha \, x^T Ax + 2\, \alpha \,  \gamma \, \|U^Tx\|_2^2} { \alpha \, x^T A x + \alpha^2 \, +\gamma \,  \| U^TL^T x\|_2^2} \,.
\end{displaymath} 
The lower bound in (\ref{eig_symm}) is obtained my minimizing the numerator and maximizing the denominator in the last expression,
keeping in mind that $\| x \|_2 = 1$ and that $\|U^T L^T \|_2^2 \le \| L^T \|_2 ^2 = \| A + \alpha I_n \|_2 = 1 + \alpha$.  Similarly, the 
upper bound is obtained by maximizing the numerator and minimizing the denominator in the last expression.
\end{proof}

We remark that the lower bound in (\ref{eig_symm}) is identical to the one in (\ref{mu}) since now $A=A^T$.  The bound suggests the
possibility of a high condition number for $(P_\alpha^S)^{-1}A_\gamma$ if $A$ is ill-conditioned (small $\lambda_{\min} (A)$), if $\alpha$
is very small, or if $\gamma $ is very large.  On the other hand, it is well known that estimates of the rate of convergence of the PCG
method based on the condition number can be very pessimistic, particularly in the presence of eigenvalue clustering. Also, we found
that the upper bound in (\ref{eig_symm}) tends to be rather loose.

\section{Numerical experiments}\label{sec6}
In this section we describe the results of numerical experiments with several matrices from the three 
application areas discussed in section \ref{sec2}.   
First we present the results of some computations aimed at assessing the quality of the eigenvalue bounds
on $Re(\lambda)$ given  in Theorem \ref{thm2},   %and Theorem \ref{thm3}, 
then we provide an evaluation of
the performance of the proposed preconditioners in terms of iteration counts and timings on a selection of test problems.
All the computations were performed using  MATLAB.R2020b on a laptop with a 4 Intel core  i7-8565U CPU @ 1.80GHz  - 1.99 GHz
and 16.0GB RAM.

\subsection{Eigenvalue bounds}
Here we consider matrices arising from the following three problems:
\begin{enumerate}
\item[(i)] A stationary Stokes problem discretized with Q2-Q1 mixed finite elements on a uniform $64\times 64$ mesh;
\item[(ii)] A stationary Oseen problem with viscosity $\nu = 0.01$ discretized with Q2-Q1 mixed finite elements on a stretched $64\times 64$ mesh;
\item[(iii)]   A Schur complement arising from a KKT system in constrained optimization;
%\item[(iv)]   A sparse-dense least squares problem.
\end{enumerate}

The first two matrices are generated using IFISS \cite{IFISS} and they are of the form $A_\gamma = A + \gamma B^T W^{-1} B$
where $A$ is the stiffness velocity matrix, $B^T$ the discrete gradient, and $W$ is the diagonal of the pressure mass matrix. Here $U^T = W^{-1/2}B$,
$n=8450$ and $k=1089$. For the Stokes problem $A$ is SPD, for the Oseen problem $A\ne A^T$ but $A+A^T$ is SPD. In both cases the 
flow problem being modeled is the 2D  leaky-lid driven cavity problem, see \cite{ESW14}.

The third matrix is a Schur complement $H + C^T Z^{-1}\Lambda C$ obtained from reduction of a KKT system in constrained optimization
\cite{MM}. Here $A=H$ is SPD, $U^T=(Z^{-1}\Lambda)^{1/2}C^T$, $\gamma =1$, $n=2500$ and $k=700$.

In all cases $A $ and $U$ have been normalized so that $\|A\|_2 = \|U\|_2 =1$.
In \Cref{T1,T2,T3} we report the minimum and maximum real part of the eigenvalues of $P_\alpha ^{-1} A_\gamma$ 
together with the value of the estimate $\mu$ in (\ref{mu}). In \Cref{T1,T2} we vary $\alpha$ and $\gamma$, 
in \Cref{T3} we fix $\gamma = 1$ and vary $\alpha$.

First we comment on the results for the linear systems arising from the incompressible Stokes and Oseen problems. 
In both cases the expression (\ref{mu}), which strictly speaking is a lower bound only for the real eigenvalues of the preconditioned
matrix, is always a lower bound for the smallest $Re(\lambda)$ (here we only show a few values of $\alpha$ and $\gamma$, but the
same was found in many more cases). We checked and found that for the preconditioned matrix
$P_\alpha ^{-1}A_\gamma$ associated with the Stokes problem, for which $A$ is SPD, the eigenvalue of smallest real part
is in fact always real; hence, (\ref{mu}) is guaranteed to be a lower bound, as is confirmed by the results in \Cref{T1}.  On the other hand, in the case
of the matrix associated with the Oseen problem, for some combinations of $\alpha$ and $\gamma$ the eigenvalue with
smallest real part was found to be non-real.  Even in these cases, however, the estimate (\ref{mu}) yielded a lower bound
on $Re(\lambda)$.  

Generally speaking we see that the lower bound is reasonably tight,
typically within an order of magnitude of the true value except in a few cases.

\begin{table}[h!]
\footnotesize
\caption{Results for eigenvalues of preconditioned $A_\gamma$ matrix from Stokes problem with $64 \times 64$ mesh and Q2-Q1  discretization.  % No diagonal scaling. 
Here $A$ is SPD. In boldface the value $\alpha = \sqrt{\gamma}$.}\label{T1}
\begin{center}
  \begin{tabular}{|c|c|c|c|c|} \hline
   $\gamma$ & $\alpha$ & $\max{(Re(\lambda))}$ & $\min{(Re(\lambda))}$& lower bound (\ref{mu}) \\
\hline
 0.1  &  0.1     &  1.818e+00  &  1.700e-02  &  5.709e-04  \\ 
      &  {\bf 0.3162}  &  1.519e+00  &  5.409e-03  &  7.250e-04  \\ 
      &  5.0     &  3.333e-01  &  3.430e-04  &  2.052e-04 \\ 
\hline 
 1.0  &  0.5     &  1.333e+00  &  6.590e-03  &  2.791e-04  \\ 
      &  {\bf 1.0}     &  1.000e+00  &  3.300e-03  &  3.140e-04  \\ 
      &  5.0     &  4.683e-01  &  6.609e-04  &  1.744e-04  \\ 
\hline
 50.0 &  1.0     &  1.532e+00  &  3.323e-03  &  1.231e-05  \\ 
      &  {\bf 7.0711}  &  1.658e+00  &  4.707e-04  &  1.928e-05  \\ 
      &  10.0    &  1.606e+00  &  3.328e-04  &  1.903e-05  \\  \hline
  \end{tabular}
\end{center}
\end{table}

\begin{table}[h!]
\footnotesize
\caption[]{Results for eigenvalues of preconditioned $A_\gamma$ matrix from Oseen problem with a stretched $64 \times 64$ mesh, $\nu=0.01$, and Q2-Q1 discretization.
Here $A\ne A^T$ but $A+A^T$ is SPD.
 In boldface the value $\alpha = \sqrt{\gamma}$.}\label{T2}
 \begin{center}
\begin{tabular}{|c|c|c|c|c|}
\hline
$\gamma$ & $\alpha$ & $\max{(Re(\lambda))}$ & $\min{(Re(\lambda))}$& lower bound (\ref{mu}) \\
\hline
 0.1  &  0.1           &  1.818e+00  &  5.317e-03  &  1.581e-04 \\ 
      &  {\bf 0.3162}  &  1.520e+00  &  1.684e-03  &  2.008e-04 \\ 
      &  5.0           &  3.333e-01  &  1.066e-04  &  5.684e-05 \\ 
\hline 
 1.0  &  0.5           &  1.333e+00  &  2.691e-03  &  7.730e-05  \\ 
      &  {\bf 1.0}     &  1.000e+00  &  1.346e-03  &  8.697e-05 \\ 
      &  5.0           &  4.423e-01  &  2.694e-04  &  4.831e-05 \\ 
\hline
 50.0 &  1.0           &  1.693e+00  &  9.185e-04  &  3.410e-06  \\ 
      &  {\bf 7.0711}  &  1.668e+00  &  1.300e-04  &  5.340e-06 \\
      &  10.0          &  1.613e+00  &  9.189e-05  &  5.271e-06 \\
\hline 
\end{tabular} 
\end{center}
\end{table}

\begin{table}[h!]
\footnotesize
\caption[]{Results for eigenvalues of preconditioned Schur complement matrix from KKT system 
(problem \texttt{mosarqp1}  from Maros and M\'esz\'aros collection). Here $A=H$ is SPD.} \label{T3}
\begin{center}
\begin{tabular}{| c | c | c | c |}
\hline
$\alpha$ &  $\max{(Re(\lambda))}$ & $\min{(Re(\lambda))}$& lower bound (\ref{mu}) \\
\hline
 0.001   &  1.998e+00  &  6.508e-03  &  7.343e-04  \\ 
 0.01    &  1.980e+00  &  6.321e-02  &  7.213e-03  \\
 0.1     &  1.818e+00  &  4.834e-01  &  6.081e-02  \\ 
 0.5     &  1.333e+00  &  8.484e-01  &  1.635e-01  \\ 
{\bf  1.0}     &  1.000e+00  &  5.384e-01  &  1.839e-01  \\ 
 5.0     &  4.335e-01  &  1.372e-01  &  1.022e-01  \\ 
 10.0    &  2.457e-01  &  7.106e-02  &  6.081e-02  \\ 
%6.081e-02
 20.0    &  1.313e-01  &  3.617e-02  &  3.337e-02  \\ 
% 50.0    &  5.470e-02  &  1.463e-02  &  1.414e-02  \\ 
\hline 
\end{tabular} 
\end{center}
\end{table}

%Finally, the fourth matrix is  from the SuiteSparse Matrix Collection \cite{SS} and is
%of the form $B_1^T B_1 + B_2^T B_2$ with $B_1$ of size $13500\times 3000$ and sparse and
%$B_2$ of size $25\times 3000$ and dense. Hence $A=B_1^TB_1$, $U^T = B_2$, $\gamma = 1$, $n=3000$ and $k=25$.

Looking at the results reported in \Cref{T3}, we see that the bound is even more accurate for this
(non PDE-related) problem. Furthermore, the eigenvalue distribution for this  test case is especially favorable
for the convergence of preconditioned iterations, suggesting fast convergence. For  this  particular problem we checked
and found that the eigenvalue of smallest real part is always real, and actually {\em all} the eigenvalues of the 
preconditioned matrix are real.

We also note that in all cases the largest value of the lower bound corresponds to $\alpha = {\sqrt \gamma}$, which is expected since the right-hand side of (\ref{mu}) 
attains its maximum for this value of $\alpha$.

%\begin{table}[htbp]
%\footnotesize
%\caption[]{Results for eigenvalues of preconditioned normal equation matrix from sparse-dense least squares
%(problem \texttt{lp\_fit2p} from SuiteSparse Matrix Collection). Here $A=B_1^TB_1$ is SPD.} \label{T4}
%\begin{center}
%\begin{tabular}{|c|c|c|c|}
%\hline
%$\alpha$ &  $\max{(Re(\lambda))}$ & $\min{(Re(\lambda))}$& lower bound (\ref{mu})\\
%\hline
% 0.001   &  1.998e+00  &  4.247e-03  &  1.597e-03  \\ 
% 0.01    &  1.980e+00  &  4.167e-02  &  1.568e-02  \\ 
% 0.1     &  1.818e+00  &  3.477e-01  &  1.322e-01  \\ 
% 0.5     &  1.333e+00  &  9.087e-01  &  3.556e-01  \\ 
 %{\bf 1.0}     &  1.000e+00  &  8.889e-01  &  4.000e-01  \\ 
% 5.0     &  5.383e-01  &  2.759e-01  &  2.222e-01  \\ 
% 10.0    &  3.180e-01  &  1.481e-01  &  1.322e-01  \\ 
% 20.0    &  1.738e-01  &  7.692e-02  &  7.256e-02  \\ 
% 50.0    &  7.350e-02  &  3.150e-02  &  3.076e-02  \\ 
%\hline 
%\end{tabular} 
%\end{center}
%\end{table}

\subsection{Test results for problems from incompressible fluid mechanics}
Here we present results obtained with the proposed approach on linear systems of the form
\begin{equation} \label{A-gamma}
 (A + \gamma B^T W^{-1} B) \, x = b 
 \end{equation}
associated with Stokes and Oseen problems. The matrices arise from Q2-Q1 discretizations 
of the driven cavity problem. We are interested in the performance of the solver with respect
to the mesh size and the parameters $\alpha$ and $\gamma$. For both Stokes and Oseen, $A$
is block diagonal but $A_\gamma = A + \gamma B^TW^{-1}B$ is not. 

\vspace{0.1in}
In our experiments we use right-preconditioned restarted GMRES with restart $m=20$ \cite{saad}.
The ideal preconditioner 
\begin{displaymath}
P_\alpha = (A + \alpha I_n) (\alpha I_n + \gamma B^TW^{-1}B)
\end{displaymath} 
is replaced by the inexact variant
\begin{equation} \label{inexact}
\tilde P_\alpha = M_{\alpha} (\alpha I_n + \gamma B^TW^{-1}B)\,,
\end{equation}
where $M_\alpha = \tilde L \tilde L^T$ is the no-fill incomplete Cholesky  (or, in the case of Oseen, incomplete LU)
factorization of $A+\alpha I_n$;
see, e.g., \cite{Benzi02}. This approximation is inexpensive in terms of cost and memory and it greatly
reduces the cost of the proposed preconditioner without adversely impacting its effectiveness.  On the other hand, the factor 
$(\alpha I  + \gamma B^TW^{-1}B)$ is inverted exactly via the SMW
formula (\ref{SMW}) with a sparse Cholesky factorization of  the $k\times k$ matrix $\alpha I_k + \gamma W^{-1/2}BB^TW^{-1/2}$.
The sparse Cholesky factorization makes use of the Approximate Minimum Degree (AMD) reordering strategy to reduce fill-in \cite{AMD}.
It is important to note that in the solution of the Navier--Stokes equations by Picard iteration, the matrices $B$ and $W$ of the Oseen problem
remain constant throughout the solution process, hence the Cholesky factorization of $\alpha I_k + \gamma W^{-1/2}BB^TW^{-1/2}$
needs to be performed only once at the beginning of the process. The matrix $A$, on the other hand, changes at each 
Picard step (since the convective term changes). Recomputing the no-fill incomplete LU factorization of $A+\alpha I_n$, however, is inexpensive.

\begin{figure}[h!]
\centering
\includegraphics[width=.45\linewidth]{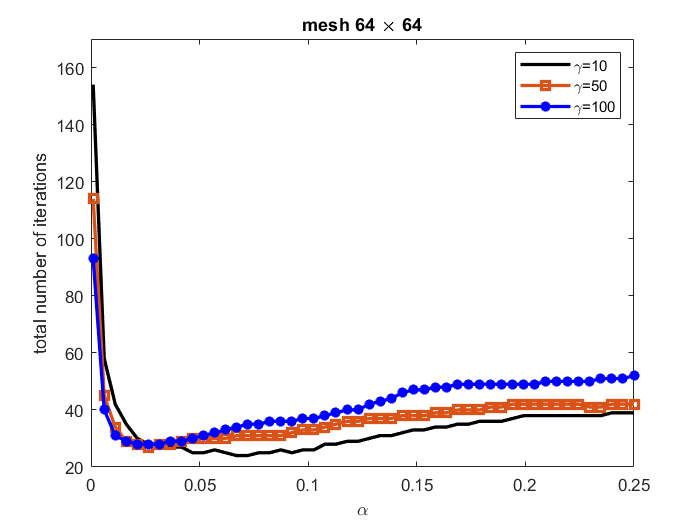}
\includegraphics[width=.45\linewidth]{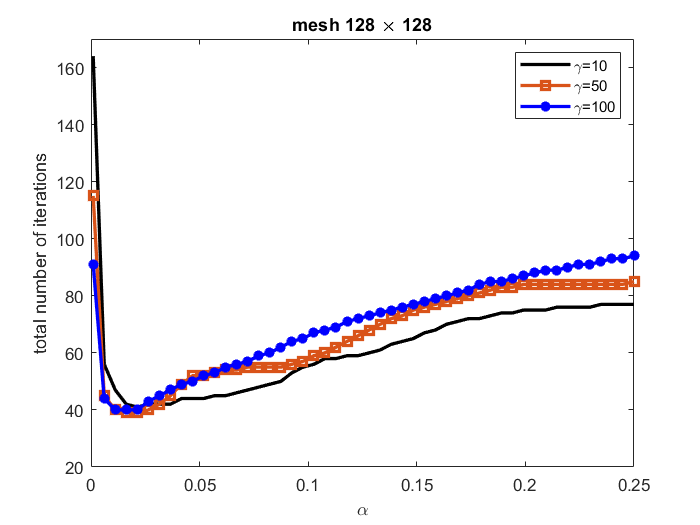}
\caption[]{Number of preconditioned iterations versus $\alpha$ for the linear systems (\ref{A-gamma}) arising from the 2D Stokes problem 
with Q2-Q1 finite element discretization on $64 \times 64$ mesh (left) and on $128 \times 128$
mesh (right) for different values of $\gamma$.} \label{FIG1}
% For the preconditioner $M_{\alpha}$ we use \texttt{ichol} with zero-fill, for $(\alpha I_k + \gamma U^T U )$ we use \texttt{chol}. Diagonal scaling. Tolerance = 1e-06}
\end{figure}

\begin{figure}[h!]
\centering
\includegraphics[width=.45\linewidth]{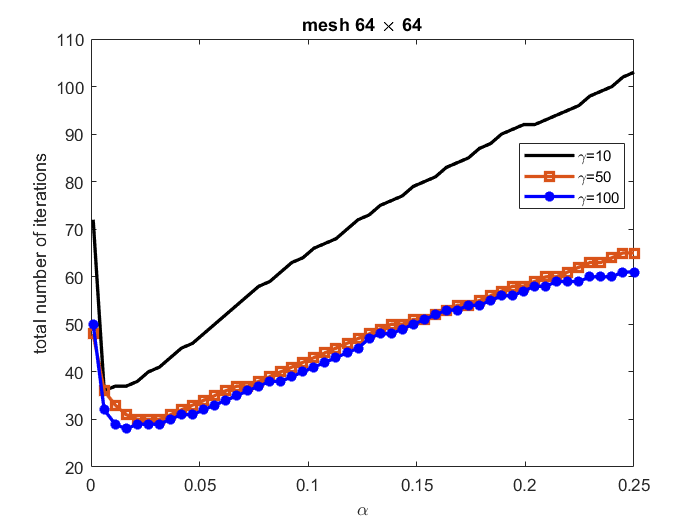}
\includegraphics[width=.45\linewidth]{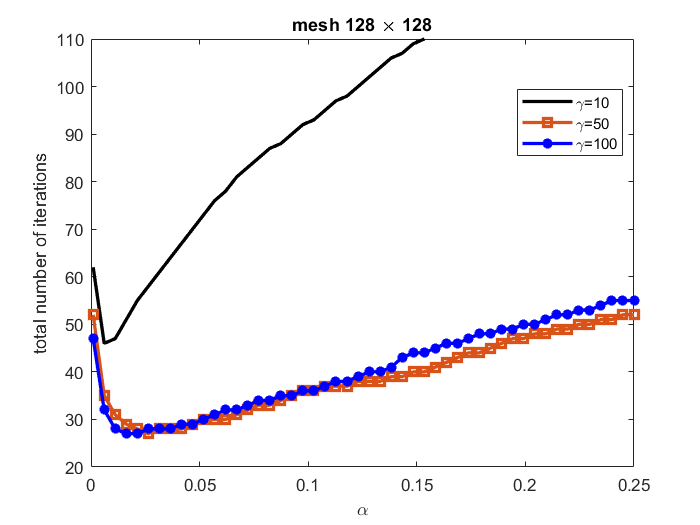}
\caption[]{Number of preconditioned  iterations versus $\alpha$ for the linear systems (\ref{A-gamma}) arising from  2D Oseen problem with $\nu = 0.01$, Q2-Q1 finite element discretization on stretched $64 \times 64$ (left) and  $128 \times 128$
 (right) meshes for different values of $\gamma$.} \label{FIG2}
\end{figure}

In \Cref{FIG1,FIG2} we show results for linear systems of the form (\ref{A-gamma}) arising from
the Stokes and Oseen problems discretized on two meshes of size $64\times 64$ and $128\times 128$, for three
different values of $\gamma$ ($=10, 50, 100$). Uniform meshes are used for the Stokes-related problem, stretched ones for
the Oseen-related one.
 We apply a symmetric diagonal scaling to $A_\gamma$ prior to 
constructing the preconditioner.
The plots show the number of right preconditioned GMRES(20) 
iterations (with the preconditioner (\ref{inexact})) as a function of the parameter $\alpha$. The stopping
criterion used is $\|b - A_\gamma x_k\|_2 < 10^{-6} \|b\|_2$, with initial guess $x_0 = 0$.  We mention that this
stopping criterion is much more stringent than the one that would be used when performing inexact preconditioner
solves in the context of the augmented Lagrangian preconditioner (\ref{AL}).

For the Stokes-related problem, the first observation is that  the fastest convergence is obtained for small values of
$\alpha$ and the number of iterations 
 is fairly insensitive to the value of $\gamma$, at least for the range of $\alpha$ values showed. 
 Also, if $\alpha$ is not too small, the curves are
 relatively flat and the number of iterations increases slowly with $\alpha$. 
 As the mesh is refined the number of iterations increases, and the optimum $\alpha$ decreases slightly. 
 
 When passing from the Stokes to the Oseen-related problem (with viscosity $\nu = 0.01$), the behavior of the solver
 is strikingly different. The convergence behavior is more sensitive to the value of $\gamma$; 
 the fastest convergence is observed for larger values of $\gamma$, for which the matrix $A_\gamma$ is more ill-conditioned.
 This is probably due to the fact that the term $\gamma B^TW^{-1}B$ becomes dominant, and the factor $\alpha I_n + \gamma B^TW^{-1}B$
 (with small $\alpha$) is a good approximation to $A_\gamma$.
 The location of the optimal value of $\alpha$ appears to be roughly the same as for the Stokes problem, but the
 curves are less flat and the number of iterations increases more rapidly as $\alpha$ moves away from the optimum.
 The most striking phenomenon, however, is that (contrary to the case of Stokes) the number of iterations appears to
 decrease as the mesh is refined. This finding is very welcome in view of the fact that the augmented Lagrangian approach is
 especially effective in the (challenging) case of the Oseen problem with small viscosity, as shown, e.g.,  in \cite{BO2006,F19}. 
 We also note that the optimal $\alpha$ is  independent of $\gamma$ when $\gamma$ is large enough.

\begin{figure}[htbp]
\centering
\includegraphics[width=1.\linewidth]{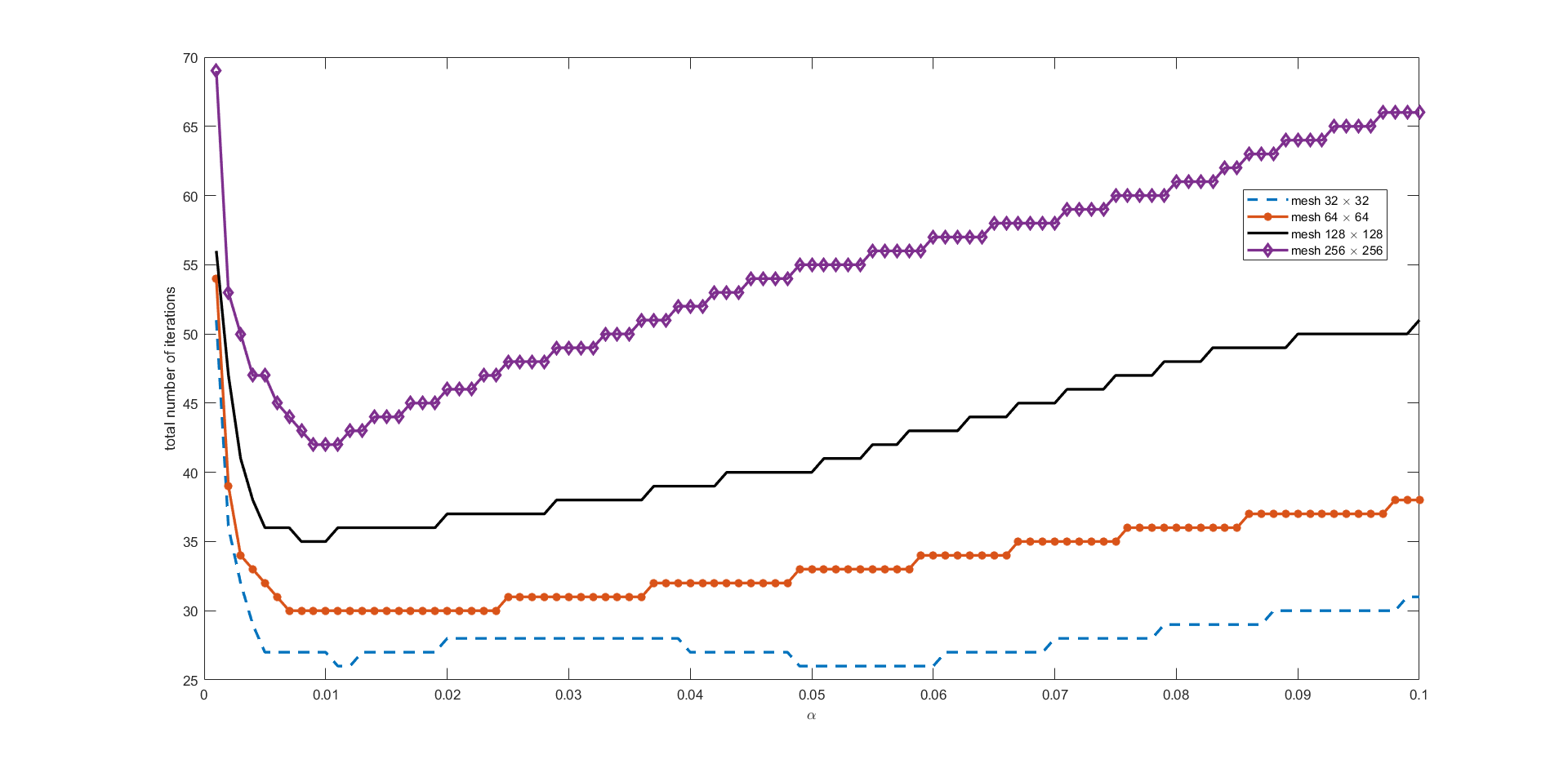}
\vspace{-0.5in}
\caption[]{Number of  iterations versus $\alpha$ for the linear systems (\ref{A-gamma}) arising from
the 2D Oseen problem with $\nu = 0.1$, $\gamma = 100$, Q2-Q1 finite 
element discretization and 
different mesh sizes. GMRES restart $m=20$, convergence residual  tolerance = $10^{-6}$. 
Diagonal scaling is applied.} \label{FIG3}
\end{figure}

\begin{figure}[htbp]
\centering
\includegraphics[width=1.\linewidth]{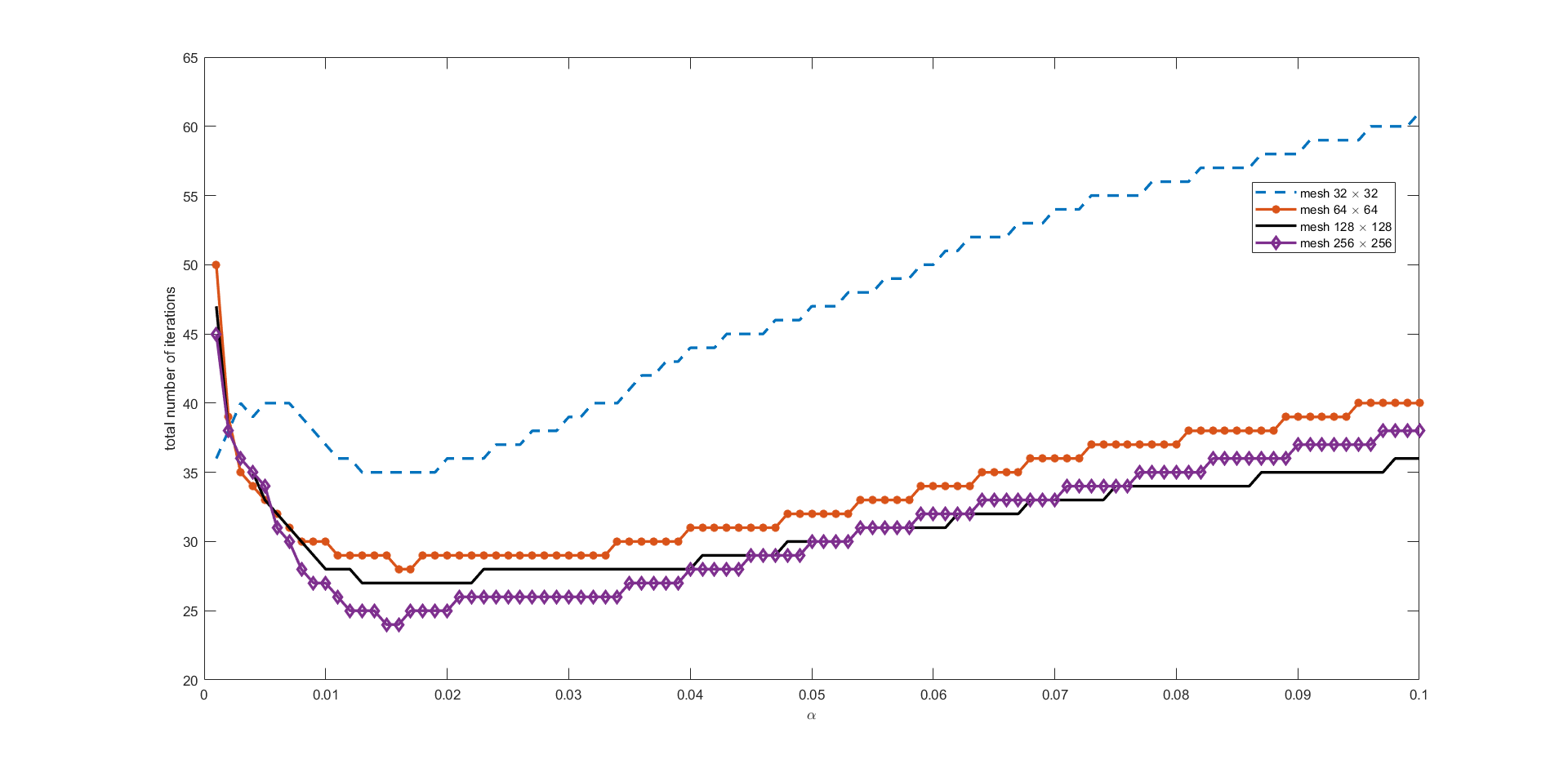}
\vspace{-0.5in}
\caption[]{Number of  iterations versus $\alpha$ for the linear systems (\ref{A-gamma}) arising from
the 2D Oseen problem with $\nu = 0.01$, $\gamma = 100$, Q2-Q1 finite 
element discretization and 
different mesh sizes. GMRES restart $m=20$, convergence residual  tolerance = $10^{-6}$. 
Diagonal scaling is applied.} \label{FIG4}
\end{figure}

\begin{figure}[htbp]
\centering
\includegraphics[width=1.\linewidth]{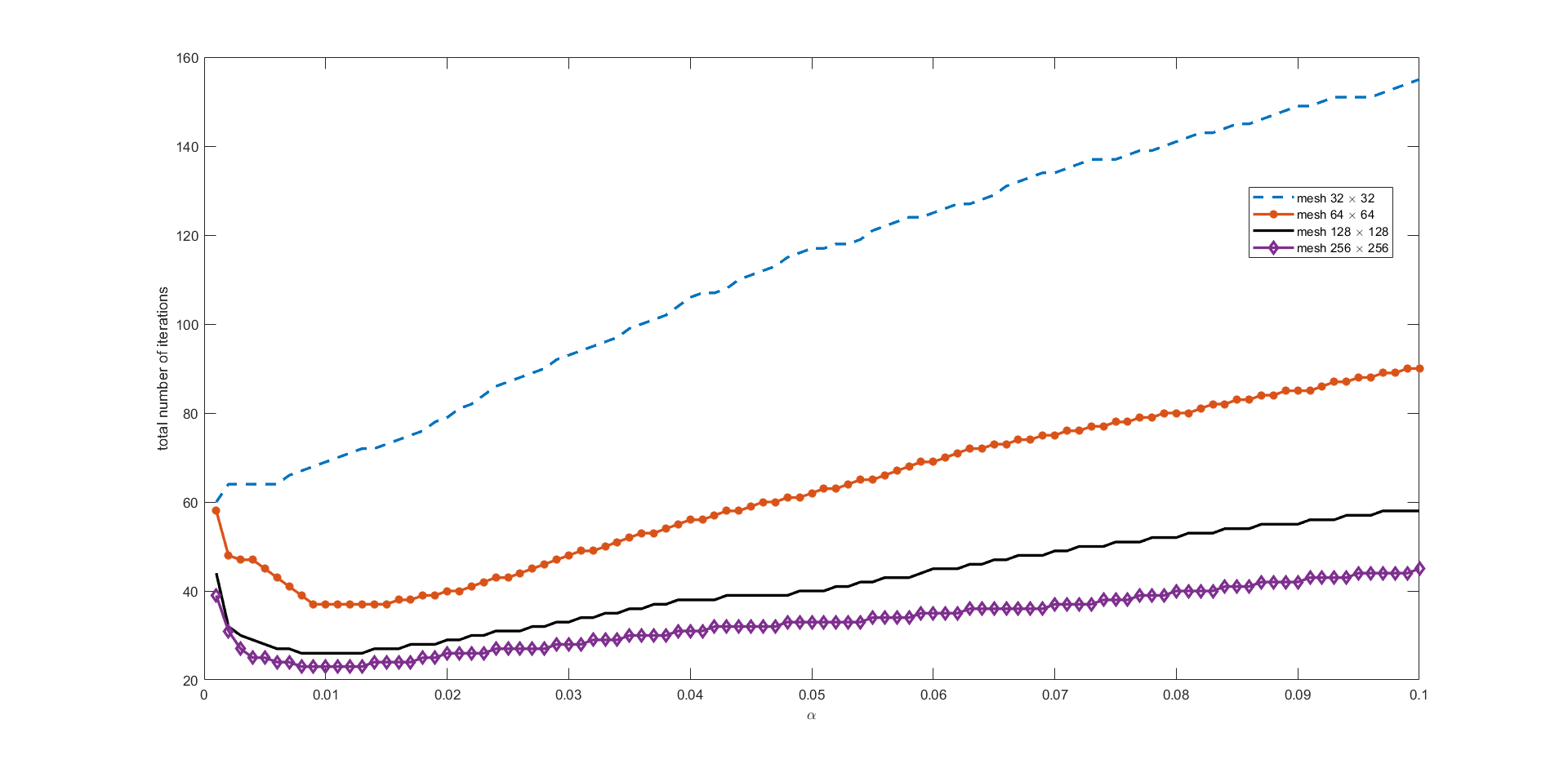}
\vspace{-0.5in}
\caption[]{Number of  iterations versus $\alpha$ for the linear systems (\ref{A-gamma}) arising from
the 2D Oseen problem with $\nu = 0.002$, $\gamma = 100$, Q2-Q1 finite 
element discretization and 
different mesh sizes. GMRES restart $m=20$, convergence residual  tolerance = $10^{-6}$. 
Diagonal scaling is applied.} \label{FIG5}
\end{figure}

 In \Cref{FIG3,FIG4,FIG5} we show results for the Oseen-related problem for three different values of the 
 viscosity $\nu$, discretized on four different (stretched) meshes. The value of $\gamma$
 is fixed at 100. Several observations are in order. The number of iterations does not seem to be 
 very sensitive to $\alpha$, as long as this is small, and the best $\alpha$ is about the same in all cases. 
 The behavior of the solver improves as the mesh is refined and as the viscosity gets smaller, i.e., the harder the 
 problem, the faster the convergence. This is especially welcome given that the augmented
 Lagrangian-based preconditioner is best employed on problems with small viscosity.

\begin{table}[htbp]
\footnotesize
\caption[]{\small{Linear system from Oseen problem with $\gamma = 100$. For each mesh we use $\alpha = 0.011$
for $\nu=0.1$, $\alpha=0.0135$ for $\nu =0.01$, and $\alpha = 0.009 $ for $\nu = 0.002$. Stretched Q2-Q1 finite element discretization. 
 Diagonal scaling is used.  $M_{\alpha}$ is ILU(0) of $A+\alpha I$, $\tilde P_\alpha = M_\alpha (\alpha I_n+ \gamma U U^T)$ with sparse
 Cholesky factorization of $k\times k$ matrix in SMW formula. M-Time and P-Time are the preconditioner construction times. Sol-Time is the time for the preconditioned iteration to achieve a relative residual norm below $10^{-6}$. All timings are in seconds.} } \label{T4}
 \begin{center}
\begin{tabular}{|c|c|c|c|c|c|c|c|c|c|}
\hline
& & & & \multicolumn{2}{c|}{ $M_{\alpha}$ } & \multicolumn{2}{c|}{ $\tilde P_{\alpha}$ } \\
\hline
$\nu$ & mesh & M-Time & P-Time  & Sol-Time & Its & Sol-Time & Its  \\
\hline
0.1 & $ 32 \times 32$ & 7.52e-04  & 1.39e-03  & 5.29e-02 & 173 &  6.75e-03& 26  \\  
& $ 64 \times 64$ & 1.22e-03  & 2.92e-03  & 3.83e-01 & 469 &  3.61e-02& 30    \\ 
& $128 \times 128 $ & 6.64e-03  & 1.54e-02  & 2.39e+00 & 603 &  2.41e-01& 36  \\  
& $256 \times 256 $ &  2.48e-02  & 8.63e-02  & 2.03e+01 & 919 &  2.21e+00& 42  \\  
\hline 
0.01&$ 32 \times 32$ &  7.26e-04  & 1.34e-03  & 1.56e-01 & 412 &  1.93e-02& 35  \\  
& $ 64 \times 64$ & 3.15e-03  & 1.48e-02  & 7.44e-01 & 466 &  5.61e-02& 29   \\ 
& $128 \times 128 $ & 1.29e-02  & 3.13e-02  & 3.05e+00 & 493 &  3.31e-01& 27  \\  
& $256 \times 256 $ &   3.95e-02  & 1.38e-01  & 1.52e+01 & 486 &  1.53e+00& 25   \\  
\hline 
0.002 & $ 32 \times 32$ & 4.86e-04  & 9.36e-04  & 1.37e-01 & 754 &  1.42e-02& 68 \\  
& $ 64 \times 64$ & 1.35e-03  & 3.05e-03  & 4.89e-01 & 522 &  4.85e-02& 37 \\ 
& $128 \times 128 $ & 6.16e-03  & 1.51e-02  & 4.09e+00 & 1037 &  1.80e-01& 26 \\  
& $256 \times 256 $ &  2.44e-02  & 8.26e-02  & 1.52e+01 & 767 &  1.06e+00& 23  \\  
\hline 
\end{tabular}
\end{center}
\end{table}

In \Cref{T4} we report iteration counts and timings for the Oseen-related problems with
the three values of $\nu = 0.1, 0.01, 0.002$ for different mesh sizes and $\gamma = 100$.
For completeness we also include results obtained using the 
simple no-fill  incomplete factorization $\tilde L\tilde U \approx A + \alpha I_n$ as a preconditioner for the system (\ref{A-gamma}). Not surprisingly, this preconditioner yields
very slow convergence, showing the importance of including the $\gamma$-dependent term in the preconditioner. Note that the cost of
forming the preconditioner (\ref{inexact}) is quite low, only slightly higher than the cost of $M_\alpha$, and that the iterative solution time dominates the
overall cost.  The results confirm the effectiveness of the proposed preconditioner, especially for small values of $\nu$
and finer meshes.  We also note that the cost of the preconditioner construction is low compared to the overall solution costs, and it 
is dominated by the Cholesky  factorization of the $k\times k$ matrix in the SMW formula. As already mentioned, when solving the Navier-Stokes equations  this factorization
needs to be performed only once since the matrix being factored does not change in the course of the Picard iteration.

\subsection{Test results on Schur complements from KKT systems}
Here we present the results of some tests on two linear systems of the form (\ref{KKT}).
In the first problem (\texttt{stcqp2} from \cite{MM}) we have $n=4097$, $k=2052$. The Schur
complement  matrix $H + C^T(Z^{-1}\Lambda) C$ has condition number $2.63 \times 10^4$.  
%We note that here $H$ is symmetric positive semidefinite (singular). 

In the second problem (\texttt{mosarqp1} from \cite{MM}) we have $n=2500$, $k=700$,
and the condition number of $H + C^T(Z^{-1}\Lambda) C$  is $3.35 \times 10^4$.  

We report results for GMRES(20) with the inexact preconditioner $\tilde P_\alpha 
=M_\alpha (\alpha I_n + C^T(Z^{-1}\Lambda) C)$, where $M_\alpha = \tilde L \tilde L^T$ is the
no-fill incomplete Cholesky factorization of $H + \alpha I_n$, as well as for  CG preconditioned with $M_\alpha$ and
with the symmetrized preconditioner $\tilde P_\alpha^S := \tilde L (\alpha I_n + C^T(Z^{-1}\Lambda) C)\tilde L^T$.
In the tables, an entry  `$2000^*$' means that the stopping criterion was not met after 2000 iterations.

The results are shown in \Cref{T5,T6}. We can see that for fast convergence of the preconditioned
iterations, larger values of $\alpha$ must be used compared to the previous set of test problems,
especially with CG.
We also see that the performance of GMRES with the unsymmetric preconditioner $\tilde P_\alpha$
is generally better than the performance of CG with the symmetrized preconditioner $\tilde P_\alpha^S$. For
both problems, preconditioning only with $M_\alpha = \tilde L \tilde L^T \approx H+\alpha I_n$ is ineffective.
We mention that for these problems, diagonal scaling prior to computing the preconditioner
led to worse performance in some cases and was generally not beneficial.

\begin{table}[t!]
\footnotesize
\caption[]{ Iteration counts for \texttt{stcqp2}. 
Left: GMRES(20).  Right: CG % For the preconditioner $M_{\alpha}$ we use \texttt{ichol} with zero-fill, for $(\alpha I_k + \gamma U^T U)$ we use \texttt{chol}. No diagonal  scaling. $A = diag(0,1,....,1)$}
} \label{T5}
\begin{center}
\begin{tabular}{|c|c|c|c|}
\hline
$\alpha$ &  $M_{\alpha}$ & $\tilde P_{\alpha}$ & no prec. \\
\hline
% 0.001 & 1077 & 2000 & 1260    \\ 
% 0.01 & 900 & 2000 &   \\ 
% 0.1 & 820 & 2000 &   \\ 
% 0.5 & 1020 & 240 &   \\ 
 1.0 & 873 & 159 & 1260  \\ 
 10.0 & 446 & 46 &   \\ 
 20.0 & 732 & 34 &   \\ 
 30.0 & 785 & 33 &   \\ 
 40.0 & 776 & 36 &   \\ 
 50.0 & 809 & 38 &   \\ 
 70.0 & 854 & 40 &   \\ 
 100.0 & 853 & 42 &  \\ 
% 130.0 & 834 & 45 &  \\ 
 %150.0 & 861 & 46 &  \\ 
% 190.0 & 660 & 49 &  \\ 
% 220.0 & 620 & 52 &  \\ 
% 260.0 & 701 & 56 &  \\ 
% 300.0 & 736 & 61 &  \\ 
\hline 
\end{tabular}
 \quad
\begin{tabular}{|c|c|c|c|}
\hline
$\alpha$ &  $M_{\alpha}$ & $\tilde P_{\alpha}^S$ & no prec. CG \\
\hline
% 0.001 & 2000 & 278  \\ 
% 0.01 & 2000 &   \\ 
% 0.1 & 2000 &   \\ 
 %0.5 & 2000 &   \\ 
 1.0   & 236 & 2000* &  278  \\ 
% 10.0  & 213 & 919   &       \\ 
 20.0  & 229 & 485   &       \\ 
% 30.0  & 227 & 334   &       \\ 
% 40.0  & 201 & 257   &       \\ 
50.0  & 241 & 210   &       \\ 
% 70.0  & 249 & 153   &       \\ 
 100.0 & 248 & 111   &       \\ 
% 130.0 & 253 & 89    &       \\ 
 150.0 & 250 & 80    &       \\ 
% 190.0 & 254 & 80    &       \\ 
 220.0 & 259 & 79    &       \\ 
 260.0 & 261 & 83    &       \\ 
 300.0 & 256 & 83    &       \\ 
\hline 
\end{tabular}
\end{center}
\end{table}

\begin{table}[t!]
\footnotesize
\caption[]{Iteration counts for \texttt{mosarqp1}. 
Left: GMRES(20).  Right: CG.%No diagonal scaling is applied. (LEFT) PGMRES.  (RIGHT) PCG. We consider a symmetrized version of the preconditioner: $P_{\alpha}^S = L(\alpha I_n + \gamma U U^T) L^T $, %where $L$ is the no-fill  incomplete Cholesky factor of $H + \alpha I_n$.}
} \label{T6}
\begin{center}
\begin{tabular}{|c|c|c|c|}
\hline
$\alpha$ &  $M_{\alpha}$ & $\tilde P_{\alpha}$ & no prec. \\
\hline
% 0.001  & 2000*  & 309 & 2000*  \\ 
 0.01   & 2000*  & 66  &        \\ 
 0.1    & 2000*  & 20  &        \\ 
% 0.5    & 2000*  & 9   &        \\ 
 1.0    & 2000*  & 6   &        \\ 
 10.0   & 2000*  & 11  &        \\ 
 20.0   & 2000*  & 13  &        \\ 
 30.0   & 2000*  & 16  &         \\ 
\hline 
\end{tabular} \quad
\begin{tabular}{|c|c|c|c|}
\hline
$\alpha$ & $M_{\alpha}$ &  $\tilde P_{\alpha}^S$ & no prec. CG \\
\hline
% 0.001 & 225 & 439  &  246  \\ 
 0.01  & 225 & 178  &    246   \\ 
 0.1   & 225 & 60   &       \\ 
% 0.5   & 228 & 26   &       \\ 
 1.0   & 228 & 18   &       \\ 
 10.0  & 244 & 15   &       \\ 
 20.0  & 244 & 15   &       \\ 
 30.0  & 244 & 19   &       \\ 
\hline 
\end{tabular}
\end{center}
\end{table}

\subsection{Test results for sparse-dense least squares problems}
Finally, we present some results for three linear systems of the form (\ref{normal}) stemming 
from the solution of sparse-dense least squares problem. 

%The first test problem is \texttt{lp\_fit2p} from \cite{SS}. Here $B_1$ is $13500\times 3000$, 
%$B_2$ is $25\times 3000$ (hence $n=3000$, $k=25$). The normal equations are ill-conditioned:
%  $\kappa_2 (B^TB)= \kappa_2 (B_1^TB_1 + B_2^TB_2) = 2.52\times 10^9$. Furthermore, 
%the 2-norm of $B_2^TB_2$ is 7 orders of magnitude larger than that of $B_1^TB_1$. No diagonal scaling is applied
%to this problem.

The first test problem, \texttt{scfxm1-2r} is  from \cite{SS}. Here $B_1$ is $65886\times 37980$, $B_2$ is $57\times 37980$ (so
$n=37980$, $k=57$),  $\kappa_2 (B_1^TB_1 + B_2^TB_2) = 9.32\times 10^6$. We note that $B_1$ is rank deficient,
hence $A = B_1^TB_1$ is singular. Diagonal scaling is applied here. 

The second problem is \texttt{neos}, again from \cite{SS}.  Here $m = 515905$, $n=479119$ and $k = 2708$.
No diagonal scaling was applied to this problem.

The third and largest test problem,  \texttt{stormg2-1000}, is taken from \cite{NR}. 
Here $B_1$ is $1377185\times 528185$, 
$B_2$ is $121\times 528185$ (hence we have $n=528185$, $k=121$).  No diagonal scaling is used on this matrix.
%{\bf Note:} the norm of $B_2^TB_2$ is 7 orders of magnitude larger than that of $B_1^TB_1$.

%\begin{table}[htbp]
%\footnotesize
%\caption[]{Iteration counts for  \texttt{lp\_fit2p}.  Left: GMRES(20).  Right: CG. %$M_\alpha$ is replaced by
%%its no-fill incomplete Cholesky approximation. No diagonal scaling is used.}
%} \label{T8}
%\begin{center}
%\begin{tabular}{|c|c|c|c|}
%\hline
%$\alpha$ &  $M_{\alpha}$ & $\tilde P_{\alpha}$ & no prec. \\
%\hline
% 0.001  &  109  &  10  &  174  \\ 
% 0.01   &  109  &  10  &       \\ 
% 0.1    &  109  &  10  &       \\ 
% 0.5    &  110  &  10  &       \\ 
 %1.0    &  110  &  8   &       \\ 
% 10.0   &  139  &  7   &       \\ 
% 100.0  &  160  &  17  &       \\ 
% 110.0  &  163  &  17  &       \\ 
%\hline 
%\end{tabular} \quad
%\begin{tabular}{|c|c|c|c|}
%\hline
%$\alpha$ & $M_{\alpha}$  & $\tilde P_{\alpha}^S$ & no prec. CG \\
%\hline
% 0.001 & 64  & 1137  &   88    \\ 
% 0.01  & 72  & 747   &         \\ 
% 0.1   & 72  & 471   &         \\ 
% 0.5   & 68  & 297   &         \\ 
% 1.0   & 77  & 219   &         \\ 
% 10.0  & 85  & 63    &         \\ 
% 100.0 & 85  & 33    &         \\ 
% 110.0 & 86  & 35    &         \\ 
%\hline 
%\end{tabular}
%\end{center}
%\end{table}

\begin{table}[h!]
\footnotesize
\caption[]{Iteration counts for \texttt{scfxm1-2r} problem. Left: GMRES(20). Right: CG.} \label{T7}
\begin{center}
\begin{tabular}{|c|c|c|c|}
\hline
$\alpha$ &  $M_{\alpha}$ & $\tilde P_{\alpha}$ & no prec. \\
\hline
 0.001  &  1572 &  555 &  240  \\ 
 0.01   &  693  &  91  &       \\ 
 0.1    &  183  &  36  &       \\ 
 0.5    &  154  &  39  &       \\ 
 1.0    &  155  &  50  &       \\ 
 10.0   &  213  &  141 &       \\ 
\hline 
\end{tabular} \quad
\begin{tabular}{|c|c|c|c|}
\hline
$\alpha$ &  $M_{\alpha}$  &  $\tilde P_{\alpha}^S$ & no prec. CG \\
\hline
 0.001 & 198  & 2000* &  180  \\ 
 0.01  & 171  & 1066  &       \\ 
 0.1   & 130  & 349   &       \\ 
 0.5   & 129  & 123   &       \\ 
 1.0   & 134  & 96    &       \\ 
 10.0  & 169  & 105   &       \\ 
\hline 
\end{tabular}
%Diagonal scaling is applied. . }
\end{center}
\end{table}

\begin{table}[h]
\footnotesize
\caption[]{Iteration counts for \texttt{neos} problem. Left: GMRES(20). Right: CG.%Diagonal scaling is applied. }
} \label{T8}
\begin{center}
\begin{tabular}{|c|c|c|c|}
\hline
$\alpha$ &  $M_{\alpha}$ & $\tilde P_{\alpha}$ & no prec. \\
\hline
% 0.01    & 2000* &  1092  & 1638   \\ 
 0.1     & 2000* &  280   &   1638     \\ 
 1.0     & 1641 &  53    &        \\ 
 5.0     & 1585 &  34    &        \\ 
 10.0    & 913  &  32    &        \\ 
 20.0    & 1404 &  34    &        \\ 
 30.0    & 1272 &  38    &        \\ 
% 50.0    & 1655 &  50    &        \\ 
\hline 
\end{tabular} \quad
\begin{tabular}{|c|c|c|c|}
\hline
$\alpha$ & $M_{\alpha}$ &  $\tilde P_{\alpha}^S$ & no prec. CG \\
\hline
 1.0    & 496 & 2000* &   325   \\ 
 5.0    & 409 & 974   &         \\ 
 10.0   & 371 & 466   &         \\ 
% 30.0   & 352 & 153   &         \\ 
% 50.0   & 339 & 101   &         \\ 
 100.0  & 342 & 83    &         \\ 
 120.0  & 338 & 81    &         \\ 
 150.0  & 342 & 87    &         \\ 
\hline 
\end{tabular}
\end{center}
\end{table}

\begin{table}[h!]
\footnotesize
\caption[]{Iteration counts for \texttt{stormg2-1000}.  Left: GMRES(20). Right: CG. %No diagonal scaling is used.}
} \label{T9}
\begin{center}
\begin{tabular}{|c|c|c|c|}
\hline
$\alpha$ &  $M_{\alpha}$ & $\tilde P_{\alpha}$ & no prec. \\
\hline
 0.001  &  2000*  &  2000*  &  2000*  \\ 
 0.01   &  2000*  &  334 &       \\ 
 0.1    &  2000*  &  98  &       \\ 
 0.5    & 2000*   &   43   & \\
1.0     &  2000*   &  50  &   \\    
 5.0    & 2000* &  89   &       \\ 
 10.0   &  2000*  &  118   &       \\ 
% 20.0   &  2000*  &  176   &       \\ 
\hline 
\end{tabular} \quad
\begin{tabular}{|c|c|c|c|}
\hline
$\alpha$ & $M_{\alpha}$ & $\tilde P_{\alpha}^S$ & no prec. CG \\
\hline
 100.0  & 2000* & 2000* & 2000* \\
 110.0  & 2000* & 1906  &       \\ 
 600.0  & 2000* & 613   &       \\
 1200.0 & 2000* & 514   &       \\
 1600.0 & 2000* & 470   &       \\
 1800.0 & 2000* & 497   &       \\
 2000.0 & 2000* & 507   &       \\ 
\hline 
\end{tabular}
\end{center}
\end{table}
%\vspace{-0.128in}

%We remark that the cost for $P_\alpha$ is only slightly larger than for $M_\alpha$. 
%\begin{table}[htbp]
%\footnotesize
%\caption[]{GMRES(20) and CG iterations and timings (secs.)~for  \texttt{scfxm1-2r} 
%problem. \\

As in all previous tests, we do not form the coefficient matrix explicitly but we perform
sparse matrix-vector products with $B_1$, $B_2$ and their transposes. We present results for 
GMRES with the preconditioner $\tilde P_\alpha = M_\alpha (\alpha I_n + B_2^T B_2)$, where
$M_\alpha =\tilde L\tilde L^T$ is given by the no-fill incomplete Cholesky factorization of $B_1 ^TB_1 +\alpha I_n$, and for CG
with the symmetrized preconditioner $\tilde P_\alpha^S = \tilde L (\alpha I_n + B_2 ^TB_2)\tilde L^T$.
The inversion of $(\alpha I_n + B_2^T B_2)$ via the SMW formula is inexpensive, as it requires computing
a $k\times k$ dense Cholesky factorization with small $k$. Computing the incomplete Cholesky factor of $B_1^TB_1 + \alpha I_n$
is also very cheap.

The results are presented in \Cref{T7,T8,T9}. We see again that for appropriate values of $\alpha$ the convergence is
fast, especially for GMRES with the nonsymmetric version of the preconditioner. As in the case of the Schur complement
systems from constrained optimization, and unlike the case of incompressible flow problems, the optimal $\alpha$ is often
relatively large. One should keep in mind that these matrices have entries of very different magnitude from the ones encountered
in finite element problems, so the scaling is very different, as is the relative size of the two terms $A$ and $UU^T$.

\newpage

%$M_{\alpha}$ is IC(0) of $A=B_1^TB_1+\alpha I$. No diagonal scaling is needed.  In PGMRES, we do not form the matrix 
%$A = B_1^T B_1$, but we compute mat-vecs as $B_1^T (B_1 x)$. Here $B_1$ is $1377185\times 528185$, 
%$B_2$ is $121\times 528185$ (hence we have $n=528185$, $k=121$).  
%{\bf Note:} the norm of $B_2^TB_2$ is 7 orders of magnitude larger than that of $B_1^TB_1$.

\begin{table}[t]
\footnotesize
\caption[]{GMRES(20) iterations and timings (secs.)~for \texttt{scfxm1-2r} problem.
 } \label{T10}
 \begin{center}
\begin{tabular}{|c|c|c|c|c|c|c|c|c|}
\hline
& &  & \multicolumn{2}{c|}{ $M_{\alpha}$ } & \multicolumn{2}{c|}{ $\tilde P_{\alpha}$ }& \multicolumn{2}{c|}{ $\tilde P_{\alpha}^S$ } \\
\hline
$\alpha$ & M-Time & P-Time  & Sol-Time & Its & Sol-Time & Its & Sol-Time & Its  \\
\hline
0.1 & 6.37e-03 & 6.47e-03 & 4.34e-01 & 183 & 8.78e-02 & 36  & 6.29e-01   & 349 \\ 
1.0 & 6.62e-03 & 6.68e-03 & 3.76e-01 & 155 & 1.28e-01 & 50  &  1.76e-01   & 96 \\ 
1.5 & 6.77e-03 & 6.83e-03 & 3.70e-01 & 155 & 1.50e-01 & 61  & 1.71e-01   & 96 \\ 
\hline 
\end{tabular}
\end{center}
\end{table}

\begin{table}[h!]
\footnotesize
\caption[]{GMRES(20) iterations and timings (secs.)~for  \texttt{stormg2-1000} 
problem. 
%$M_{\alpha}$ is IC(0) of $A=B_1^TB_1+\alpha I$. No diagonal scaling is needed.  In PGMRES, we do not form the matrix 
%$A = B_1^T B_1$, but we compute mat-vecs as $B_1^T (B_1 x)$. Here $B_1$ is $1377185\times 528185$, 
%$B_2$ is $121\times 528185$ (hence we have $n=528185$, $k=121$).  
%{\bf Note:} the norm of $B_2^TB_2$ is 7 orders of magnitude larger than that of $B_1^TB_1$.
 } \label{T11}
 \begin{center}
\begin{tabular}{|c|c|c|c|c|c|c|c|c|}
\hline
& &  & \multicolumn{2}{c|}{ $M_{\alpha}$ } & \multicolumn{2}{c|}{ $\tilde P_{\alpha}$ } \\
\hline
$\alpha$ & M-Time & P-Time  & Sol-Time & Its & Sol-Time & Its  \\
\hline
0.5 & 7.05e-02 & 7.09e-02 & 1.01e+02 & 2000* & 2.28e+00 & 43    \\ 
1.0 & 6.90e-02 & 6.93e-02 & 1.01e+02 & 2000* & 2.61e+00 & 50    \\ 
1.5 & 6.93e-02 & 6.96e-02 & 1.01e+02 & 2000* & 3.10e+00 & 59   \\ 
\hline 
\end{tabular}
 \end{center}
\end{table}

In \Cref{T10,T11} we report some timings for the test problems  \texttt{scfxm1-2r} and \texttt{stormg2-1000}. 
We remark that the cost for constructing the preconditioner  $\tilde P_\alpha$
(or $\tilde P_\alpha^S$) is only slightly higher than for $M_\alpha$,
and is negligible compared to the cost of the iterative solution phase.

\section{Conclusions}\label{sec7}
In this paper we have proposed and investigated some approaches for solving large linear systems of the form $(A + \gamma UU^T) \, x = b$.
Such linear systems arise in several applications and can be challenging due to possible ill-conditioning and the fact that the coefficient
matrix $A_{\gamma} = A + \gamma UU^T$ often cannot be formed explicitly. We have proposed a preconditioning technique for use with GMRES,
together with a symmetric variant which can be used with the CG method when $A=A^T$. Some bounds on the eigenvalues of the preconditioned
matrices have been obtained.   Numerical experiments on a variety of test problems from different application areas indicate that the proposed approach is quite robust
 and can yield very fast convergence even when applied inexactly. In some cases we have been able to describe a heuristic for estimating the optimal
 value of the parameter $\alpha$ that appears in the preconditioner. 
 
 Future work should focus on obtaining better estimates of the preconditioned spectra and on heuristics for the choice of $\alpha$ for 
 general problems.  For PDE-related problems, estimates of the optimal  $\alpha$ could be obtained based on a Local Fourier Analysis, as done
 for other preconditioners (e.g., \cite{RDF}). Also, we plan to investigate the use of the preconditioner in the context of augmented Lagrangian preconditioning
 of incompressible flow problems, in order to determine how accurately one needs to solve the system (\ref{A-gamma}) at each appplication of
 the block triangular preconditioner (\ref{AL}) without adversely impacting the performance of FGMRES. 
 
 Finally,  for SPD problems the use of CG with the symmetrized variant of the preconditioner generally led to worse
 results (in terms of solution times)  than the use of restarted GMRES with the nonsymmetric preconditioner. Hence, the
 question of how to best symmetrize the preconditioner when $A$ is symmetric  remains open.

\section*{Acknowledgments}
The authors would like to thank Carlo Janna and two anonymous reviewers for helpful comments.

\bibliographystyle{siamplain}

\end{document}